\documentclass[11pt]{article}
\usepackage[]{times}
\usepackage{amsmath}
\usepackage{paralist}
\usepackage{fullpage}
\usepackage{amssymb}
\usepackage{mathrsfs}
\usepackage{bm}
\usepackage{subfig}
\usepackage{graphicx}
\usepackage{graphicx}
\usepackage{wrapfig}
\usepackage{epsfig}
\usepackage{url}
\usepackage{fullpage}
\usepackage{psfrag}
\usepackage{verbatim}
\usepackage{amsthm}
\usepackage[]{times}
\usepackage{amsmath}
\usepackage{paralist}
\usepackage{fullpage}
\usepackage{amssymb}
\usepackage{amsrefs}

\usepackage[colorlinks=true, pdfstartview=FitV, linkcolor=blue,
            citecolor=blue, urlcolor=blue]{hyperref}
\usepackage[usenames]{color}
\definecolor{Red}{rgb}{0.7,0,0.1}
\definecolor{Green}{rgb}{0,0.6,0}

\usepackage{mathrsfs}
\usepackage{bm}
\usepackage{graphicx}
\usepackage{graphicx}
\usepackage{wrapfig}
\usepackage{epsfig}
\usepackage{url}
\usepackage{fullpage}
\usepackage{psfrag}
\usepackage{verbatim}

\numberwithin{equation}{section}

\newtheorem{prop}{Proposition}[section]
\theoremstyle{remark}
\newtheorem{remark}{Remark}[section]

\title{Errata to Stochastic explosion and non-uniqueness for 
$\alpha$-Riccati equation}
\author{Radu Dascaliuc\thanks{Department of Mathematics,  Oregon State University, Corvallis, OR, 97331. {dascalir@math.oregonstate.edu}}
\and
Tuan N.\ Pham
 \thanks{Department of Mathematics,  Eastern Oregon
 University, La Grande, OR 97850. {tnpham@eou.edu}}
\and Enrique Thomann\thanks{Department of Mathematics,  Oregon State University, Corvallis, OR, 97331. {thomann@math.oregonstate.edu}}
\and
Edward C.\ Waymire\thanks{Department of Mathematics,  Oregon State University, Corvallis, OR, 97331. {waymire@math.oregonstate.edu}.}
}

\begin{document}

\maketitle

\begin{abstract}
An error occurs in a part of the statement and proof of Proposition 2.2 in the paper \emph{``Stochastic explosion and non-uniqueness for $\alpha$-Riccati equation''}, Jour. Math. Anal. and Appl., 476, (2019), 53-85 that is corrected in this errata. The revised result reveals a new and unexpected critical phenomena, having further implications for non-uniqueness of solutions to a nonlinear differential equation of the Riccati type.  

\end{abstract}

\section{Introduction}
\label{introsec}
The Proposition 2.2 in \cite{alphariccati} applies to
the $\alpha$-Riccati model, but the proof 
 is only valid
for $\alpha\in[0,1]\cup[2,\infty)$.  The error is in the
second displayed equation following (2.12)
 on p. 61 of \cite{alphariccati}.  The estimates prior to
 this line are valid and are used in this correction. 
 
Let us first recall some notations introduced in \cite{alphariccati}. The \emph{set of $t$-leaves} on a binary tree $\mathcal{I}=\{\emptyset\}\cup\cup_{n\in\mathbb{N}}\{1,2\}^n$ is defined as
$$V(t) = \left\{v\in\mathcal{I}: \sum_{j=0}^{|v|-1}\alpha^{-j}T_{v|j}
\le t < \sum_{j=0}^{|v|}\alpha^{-j}T_{v|j}\right\},\ t>0, 
\quad V(0) =\{\emptyset\}.$$
The \emph{shortest path} $S$ and the \emph{longest path} $L$ of the $\alpha$-Riccati cascade $\{\alpha^{-|v|}T_v\}_{v\in\mathcal{I}}$ are defined as
$$S=\lim_{n\to\infty}\min_{|v|=n}\sum_{j=0}^{n}\alpha^{-j}T_{v|j},\ \ \ L=\lim_{n\to\infty}\max_{|v|=n}\sum_{j=0}^{n}\alpha^{-j}T_{v|j}.$$
The $\alpha$-Riccati cascade is said to be \emph{exploding} if $S<\infty$ a.s.\ and \emph{hyper-exploding} if $L<\infty$ a.s. The value $\alpha=1$ is where the transition between non-explosion and explosion occurs \cite{alphariccati}. Specifically, $S=\infty$ a.s.\ if $\alpha\in(0,1]$ and $S\le L<\infty$ a.s.\ if $\alpha>1$. The correction provided
below reveals a critical region $(1,2)$ of branching
rates $\alpha$ that further classifies the explosion phenomenon. Specifically, for $\alpha\in(1,2)$, the hyperexplosion is weak enough that the set of $t$-leaves is infinite with a positive probability for every $t>0$. For $\alpha\ge 2$, the hyperexplosion is strong enough that the set of $t$-leaves is a.s.\ finite for every $t>0$. As a consequence, we obtain the non-uniqueness of solutions to the $\alpha$-Riccati equation with initial data $u_0=1$ in the case $\alpha\in(1,2)$.

To confirm the theoretical results, we present in Section \ref{simsection} some numerical simulations\footnote{The
 computations took advantage of the HPC cluster at Oregon State University.} with different values of $\alpha$  to depict the tail behavior in the distributions of the number of 
 $t$-leaves and the probability of having finitely many $t$-leaves. 

\section{Corrected Statement and Proof of Proposition 2.2}
The modified proof requires an 
observation
that 
for all $\alpha>1$, $\mathbb{P}(L>t)$ 
is integrable on $(0,\infty)$ so that
${\mathbb E}L<\infty$.
This integrability property had 
previously been proven in 
\cite{athreya} for the case $\alpha\ge 2$.  

The proofs of results in Section 4 that refer to the original
 Proposition 2.2 are mitigated by a revised statement
of  Proposition 4.1. This revision leaves the main Theorems 4.2
and 4.3 of  \cite{alphariccati} intact as stated and proven. 

\begin{prop}[Proposition 2.2 Revised]\label{prop2.2rev}
Let $\alpha\ge 0$ and consider the event
$$G_t=\{\text{zero or finitely many branches of the}
\ \alpha\text{-Riccati cascade 
 crossed}\ t\}.$$
Then, for any $t>0$, 
$$
\mathbb{P}(G_t)\quad \begin{cases} = 1 \ &
 \text{if}\quad \alpha\in[0,1]\cup[2,\infty)\\
 < 1 \ & \text{if}\quad \alpha\in(1,2).
 \end{cases}
 $$

\end{prop}
\begin{proof}
That $P(G_t) = 1$ for  $\alpha\in[0,1]\cup[2,\infty)$ is 
correctly asserted and proven in \cite{alphariccati}.
In particular, for $0\le\alpha\le 1$, the assertion follows
from Part 1 of Theorem 2.1 of \cite{alphariccati}. 
For $\alpha\ge 2$,
$q(t) = \mathbb{P}(G_t^c)\le \mathbb{P}(L>t)$ is integrable
and $\lim_{t\to\infty}q(t)=0$ by \cite{athreya}.  By
 letting $t\to\infty$ in (2.12) of \cite{alphariccati}, one obtains 
 $$ \int_0^\infty q^2(t)dt = (2-\alpha)\int_0^\infty q(t)dt,$$
 from which $\int_0^\infty q^2(t)dt=0$ follows for $\alpha\ge 2$. 
 Hence, $q=0$ by continuity.
 
 For the proof that $\mathbb{P}(G_t) < 1$ for  $\alpha\in(1,2)$, we
proceed as follows.
One has $q(t) = P(G_t^c) = P(|V(t)|=\infty)$ and
$$[|V(t)|=\infty]=\cap_{n=0}^\infty C_n(t),$$
where $C_n(t)$ denotes the event that there is a path crossing the horizon
$t$ at the generation $n$ or greater.  In particular,
$C_0(t) = [L>t]$. 
Note also that the sequence of events 
$\{C_n(t)\}$ is decreasing to $\cap_{n=0}^\infty C_n(t)
= [|V(t)|=\infty]$. Let 
$$q_n(t) = \mathbb{P}(C_n(t)), \quad t>0.$$
Then $q_0(t)=\mathbb{P}(C_0(t))=\mathbb{P}(L>t)$. The sequence $\{q_n(t)\}$ is a decreasing sequence and 
$q(t) = \lim_{n\to\infty}q_n(t)$.  For $n\ge 1$, by conditioning on $T_\emptyset$, one has
\begin{equation}\label{qrecur}
q_n(t) = \int_0^t e^{-(t-s)}(2q_{n-1}(\alpha s)-q_{n-1}^2(\alpha s))ds\;.
\end{equation}
We now show that $q_0$ is integrable on $(0,\infty)$. 
By the inheritance property of the event $[L\le t]$, $q_0(t)$ satisfies (2.10) of  \cite{alphariccati}, and therefore upon integrating and a change of variables,
\begin{eqnarray}
\nonumber\alpha q_0(t)-\alpha &=& \int_0^t((2-\alpha)q_0(s)-q_0^2(s))ds
+\int_t^{\alpha t}(2q_0(s)-q_0^2(s))ds\\
\label{q0eqn}&\ge& \int_0^t((2-\alpha)q_0(s)-q_0^2(s))ds.
\end{eqnarray}
Suppose by contradiction that $\int_0^\infty q_0(s)ds=\infty$.
Since $q_0(t)<{2-\alpha\over 2}$ for
all sufficiently large $t$, say $t>t_0$, it follows that
the right-hand side of the inequality \eqref{q0eqn}, 
denoted RHS\eqref{q0eqn}, satisfies
$$\textup{RHS}\eqref{q0eqn}
\ge \int_0^{t_0}\left((2-\alpha)q_0(s)-q_0^2(s)\right)ds
+{2-\alpha\over 2}\int_{t_0}^tq_0(s)ds\to\infty
\ \text{as}\ t\to\infty.$$ 
This contradicts the fact that 
the left-hand side of  \eqref{q0eqn} satisfies $\alpha q(t)-\alpha\to-\alpha$ as $t\to\infty$. Therefore, $\int_0^{\infty}q_0(s)ds<\infty$. 
Next, suppose for purpose
of contradiction that $q(t) =0$ for all $t>0$.  Then,
$q_n(t)\downarrow 0$ as $n\to\infty$ for each $t>0$.
Since $q_n\le q_0\in L^1(0,\infty)$, one has
$\|q_n\|_1\downarrow 0$ by Lebesgue Dominated
Convergence Theorem.  From (\ref{qrecur}), one has
$$q_n(t)\le\int_0^te^{-(t-s)}2q_{n-1}(\alpha s)ds
\le\int_0^t2q_{n-1}(\alpha s)ds\le{2\over\alpha}\|q_{n-1}\|_1,\ \forall n\ge 1.$$ Thus, $\|q_n\|_\infty\le {2\over\alpha}\|q_{n-1}\|_1\to0.$ There exists $N\in\mathbb{N}$ such that 
$$\|q_n\|_\infty< 2-\alpha\quad  \forall n\ge N.$$
In particular, for any $n> N$,
$$q_n(t) = \int_0^te^{-(t-s)}q_{n-1}(\alpha s)(2-q_{n-1}(\alpha s))ds> \int_0^te^{-(t-s)}\alpha q_{n-1}(\alpha s)ds.$$ 
Taking the $L^1$-norm on both sides yields
\begin{equation}
\|q_n\|_1> \|q_{n-1}\|_1, \ n \ge 1.
\end{equation}
Thus, $\|q_n\|_1$ is an increasing sequence for $n>N$, 
contradicting the fact that $\|q_n\|_1 \downarrow 0$. Hence, $q$ is not identically zero.
To see that $q(t)>0$ for all $t$, suppose by contradiction that 
$q(t_0) = 0$ for some $t_0>0$. Then 
$$0=q(t_0) = \int_0^{t_0}e^{-(t_0-s)}(2q(\alpha s)-
q^2(\alpha s))ds.$$
By the continuity of $q$, one has $q(s)=0$ for $s\in[0,\alpha t_0]$. 
In particular, $q(\alpha t_0) =0$. Repeating this argument with
$t_0$ replaced by $\alpha t_0$ yields $q(s)=0$ for
$s\in[0,\alpha^2 t_0]$, and so on.  Thus, $q=0$ on $[0,\infty)$, which is a contradiction.
\end{proof}

In view of Proposition 2.2 Revised,  Proposition 4.1 in \cite{alphariccati} is valid as stated for $\alpha\in[0,1]\cup[2,\infty)$. 
Theorem 4.2 only involves $\alpha>{5\over 2}>2$ and therefore
remains valid as stated and proven.  Theorem 4.1
is also valid as stated in view of the following revised statement
of Proposition 4.1. 
\begin{prop}[Proposition 4.1 Revised]\label{4.1revised} Let $\alpha>0$.
\begin{enumerate}
\item[(i)]\  If $\alpha\in[0,1]\cup[2,\infty)$ then Proposition 4.1 is valid as
stated.
\item[(ii)]\  If $\alpha\in(1,2)$ then one has the following statements.
\begin{enumerate}[(a)]
\item If $X_0=0$ in iteration (4.1), 
then $X_n(t)\to\underline{X}(t)<\infty$.
\item If $X_0=1$ in
iteration (4.1), then $X_n(t)\to\overline{X}(t)\le\infty$.
Moreover, ${\mathbb E}\overline{X}(t)=\infty$ if $u_0>1$,
while $\overline{u}(t)={\mathbb E}\overline{X}(t)<\infty$ is a
solution of (2.2) if $u_0\in[0,1]$. 
\end{enumerate}
\end{enumerate}
\end{prop}
\noindent The validity of this revised statement of Proposition 4.1 holds because
it does not require the limit $M_n\to\infty$ in the comment following
Remark 4.1.

\begin{remark}
As noted at the outset, the revised Proposition 4.1, i.e.\ Proposition \ref{4.1revised}, implies that in the case 
that $\alpha\in(1,2)$, the probability
$v(t)=\mathbb{P}(G_t)$ defines a solution to the 
$\alpha$-Riccati equation corresponding to the initial data $u_0=1$ that is distinct from both $\underline{u}(t)=\mathbb{P}(S>t)$, and 
$\bar{u}(t)\equiv 1$, with 
$\underline{u}(t)<v(t)<\bar{u}(t)$ for all $t>0$.
\end{remark}
\section{Numerical Results}\label{simsection}
In this section, two sets of graphical illustrations are provided to depict behavior consistent with the newly revealed phenomena.

For $\alpha\in(1,2)$, the total number of leaves $W(t)=|V(t)|$ is infinite with a positive probability. Consequently, $\mathbb{E}W(t)=\infty$ for each $t>0$. It is conceivable that the histogram of $W(t)$ should also have a heavy tail. The hyperexplosive nature of the tree makes direct simulations of $W(t)$ challenging. To mitigate this, a \emph{truncation} of $W(t)$ is simulated. Let $V_n(t)\subset V(t)$ be the subset of vertices of heights less than or equal to $n$, and $W_n(t)=|V_n(t)|$. Then by conditioning on $T$, one may easily see that
\[W_{n}(t)\stackrel{d}{=}\left\{ \begin{array}{*{35}{c}}
   1 & \text{if} & T>t  \\
   {{W}_{n-1}^{(1)}}(\alpha (t-T))+{{W}_{n-1}^{(2)}}(\alpha (t-T)) & \text{if} & T\le t  \\
\end{array} \right.\]
and $W_n(t)\uparrow W(t)$ as $n\to\infty$. Here, $T$ is a mean-one exponential random variable, and $W_{n-1}^{(1)}$ and $W_{n-1}^{(2)}$ are two i.i.d.\ copies of $W_{n-1}$. 
Figures \ref{fig1:a}, \ref{fig2:a}, \ref{fig3:a} show the histograms of $W_{n=10}(t=2)$ with 10000 experiments in the cases $\alpha =0.66$,
$\alpha=1.5$, and $\alpha = 3$, respectively.  The heavy tail in the case $\alpha=1.5$ is consistent with the ``weak'' hyperexplosive behavior of the random tree, i.e.\ $\mathbb{P}(G_t^c)>0$. 

In the second set of figures, we simulate the function $v(t)=1-q(t)=\mathbb{P}(G_t)$, which is the probability of getting finitely many $t$-leaves. This function solves the $\alpha$-Riccati equation $v^\prime+v=v^2(\alpha t)$ with the initial data $v(0)=1$. According to \eqref{qrecur}, $v(t)$ is the pointwise limit of $v_n(t)=1-q_n(t)$ and 
\[v_n(t)=e^{-t}+\int_0^te^{-s}v^2_{n-1}(\alpha (t-s))ds,\ \ \ v_0(t)=\mathbb{P}(L<t).\]
One can interpret this recursion probabilistically as $v_n=\mathbb{E}X_n$ where
\begin{equation}\label{xrecur}X_n(t)\stackrel{d}{=}\left\{ \begin{array}{*{35}{c}}
   1 & \text{if} & T>t  \\
   {{X}_{n-1}^{(1)}}(\alpha (t-T)){{X}_{n-1}^{(2)}}(\alpha (t-T)) & \text{if} & T\le t  \\
\end{array} \right.\end{equation}
and $X_0(t)=v_0(t)=\mathbb{P}(L<t)$. To get (approximately) $v_0(t)$, observe that it is the maximal solution to the equation
\[v_0(t)=\int_0^te^{-s}v_0^2(\alpha (t-s))ds\]
and thus is the limit of the sequence $U_k(t)$, where
\begin{equation}\label{arecur}U_k(t)=\int_0^te^{-s}U^2_{k-1}(\alpha (t-s))ds\end{equation}
and $U_0(t)=1$. Now for the simulation of $v(t)$, we discretize the time interval $[0,8]$ and simulate the approximation $v_{n=10}(t)$ at each grid point via a Monte Carlo algorithm using \eqref{xrecur}, where the initial process $X_0(t)=v_0(t)$ is approximated by $U_{k=5}(t)$, obtained analytically from \eqref{arecur}. Figures \ref{fig1:b}, \ref{fig2:b}, \ref{fig3:b} illustrate $v(t)$ in the cases $\alpha=0.66$, $\alpha =1.5$, and $\alpha=3$, respectively. The dots are the numerical results obtained by Monte Carlo simulation. The orange curves are the B\'{e}zier curves fitting the data. 

Various numerical solutions to the equation were computed by Nick Hale and Andr\'{e} Weideman at Stellenbosch University, South Africa, using a different numerical method than that presented here (personal
communication). In the range $1<\alpha<2$, one of their curves has a similar profile as that for
$\mathbb{P}(G_t)$ in Figure \autoref{fig2:b}. 
Not all of these numerical solutions 
are explained by the current theory.

\bibliography{NSF2022-3}

\begin{bibdiv}
\begin{biblist}

\bib{athreya}{article}{
      author={Athreya, K.~B.},
       title={Discounted branching random walks},
        date={1985},
     journal={Adv. in Appl. Probab.},
      volume={17},
      number={1},
       pages={53\ndash 66},
}

\bib{alphariccati}{article}{
      author={Dascaliuc, Radu},
      author={Thomann, Enrique~A.},
      author={Waymire, Edward~C.},
       title={Stochastic explosion and non-uniqueness for {$\alpha$}-{R}iccati
  equation},
        date={2019},
     journal={J. Math. Anal. Appl.},
      volume={476},
      number={1},
       pages={53\ndash 85},
}

\end{biblist}
\end{bibdiv}
\begin{figure}[ht] 
    \centering
    \subfloat[]{%
        \includegraphics[scale=.55]{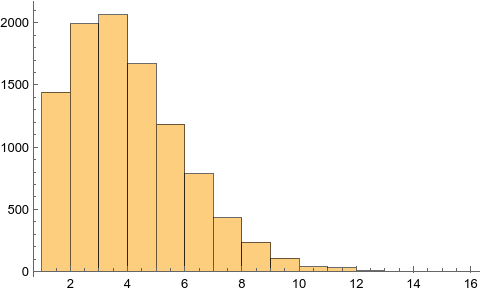}%
      \label{fig1:a}%
        }%
        \hspace{.1in}%
    \subfloat[]{%
        \includegraphics[scale=.55]{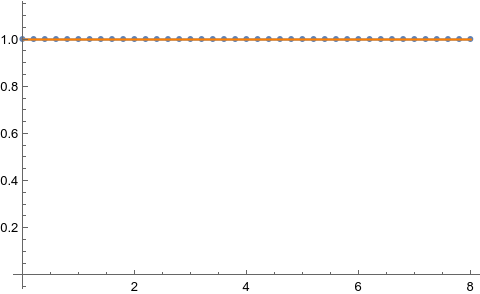}%
       \label{fig1:b}%
        }%
    \caption{The case $\alpha=0.66$}
    \label{a=0.66}
\end{figure}

\begin{figure}[ht] 
    \centering
    \subfloat[]{%
        \includegraphics[scale=.6]{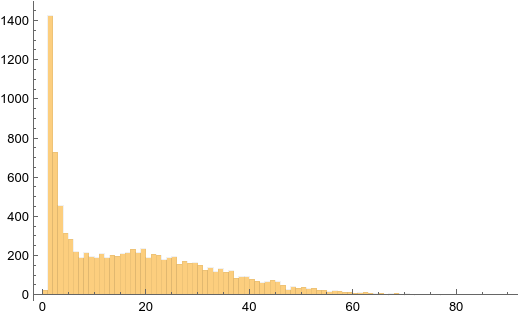}%
        \label{fig2:a}%
        }%
        \hspace{.1in}%
    \subfloat[]{%
        \includegraphics[scale=.5]{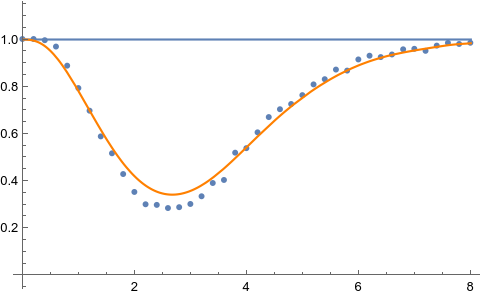}%
       \label{fig2:b}%
        }%
    \caption{The case $\alpha=1.5$}
    \label{a=1.5}
\end{figure}

\begin{figure}[ht] 
    \centering
    \subfloat[]{%
        \includegraphics[scale=.55]{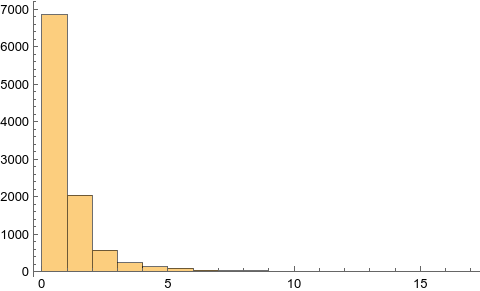}%
        \label{fig3:a}%
        }%
        \hspace{.1in}%
    \subfloat[]{%
        \includegraphics[scale=.55]{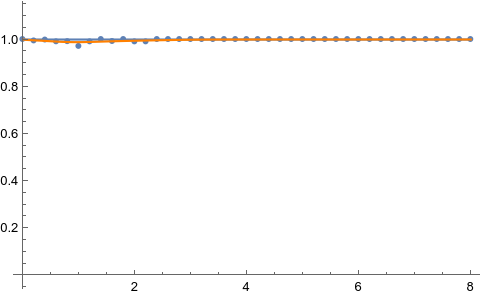}%
       \label{fig3:b}%
        }%
    \caption{The case $\alpha=3$}
    \label{a=3}
\end{figure}
\end{document}